\documentclass[11pt,a4paper]{article}

\usepackage[margin=1in]{geometry} 

\usepackage[latin1]{inputenc}
\usepackage[english]{babel}
\usepackage{amsmath}
\usepackage{color}
\usepackage{float} 
\usepackage{amsfonts}
\usepackage{amssymb,amsthm}
\usepackage{hyperref} 
\usepackage{graphicx,abstract}
\usepackage{cleveref}

\newcommand{\RR}{\mathbb{R}}

\newcommand{\NN}{\mathbb{N}}

\newtheorem{theorem}{Theorem}
\newtheorem*{theorem*}{Theorem}
\newtheorem{lemma}{Lemma}
\newtheorem{proposition}{Proposition}

\newtheorem{definition}{Definition}

\newtheorem{remark}{Remark}

\newtheorem{example}{Example}
\renewenvironment{proof}[1][]{\noindent {\bf Proof #1:\;}}{\hfill $\Box$}

\title{A note on stationarity in constrained optimization}
\author{
	Edouard Pauwels\thanks{Toulouse School of Economics, Toulouse France. Institut Universitaire de France (IUF).}
}
\begin{document}

\maketitle

\begin{abstract}
	Minimizing a smooth function $f$ on a closed subset $C$ leads to different notions of stationarity: Fr\'echet stationarity, which carries a strong variational meaning, and criticality, which is defined through a closure process and involves the notion of limiting, or Mordukovitch, subdifferential. The latter is an optimality condition which may loose the variational meaning of Fr\'echet stationarity in some settings. The purpose of this note is to illustrate that, while criticality is the appropriate notion in full generality, Fr\'echet stationarity is typical in practical scenarios. 
	
	We gather two results to illustrate this phenomenon. These results are essentially known and, our goal is to provide consize self contained arguments in the constrained optimization setting. First we show that if $C$ is semi-algebraic, then for a generic smooth semi-algebraic function $f$, all critical points of $f$ on $C$ are actually Fr\'echet stationary. Second we prove that for small step-sizes, all the accumulation points of the projected gradient algorithm are Fr\'echet stationary, with an explicit global quadratic estimate of the remainder, avoiding potential critical points that are not Fr\'echet stationary, and some bad local minima.
\end{abstract}

\section{Introduction}

We consider the problem
\begin{align}
	\label{eq:mainProblem}
	\min_{x \in C} f(x)
\end{align}
where $f \colon \RR^p \to \RR$ is $C^1$ and $C \subset \RR^p$ is closed. A point $x \in C$ is called Fr\'echet stationary for $f$ on $C$ if $f(y) - f(x) \geq o(\|y-x\|)$ for $y \in C$ (see \Cref{def:tangentCone}), and a vector $v$ is called a regular normal vector to $C$ at $x$ whenever $x$ is Fr\'echet stationary for the linear form $x \mapsto - \left\langle v, x\right\rangle$ on $C$. The notion of regular normal vector lacks basic continuity properties, and in particular limits of regular normals may not correspond to regular normals. The broader notion of criticality aims at recovering a form of continuity, see \Cref{sec:definitions}. But the price is the variational meaning of Fr\'echet stationarity which may be lost due to lack of regularity (see \Cref{ex:sparseVectors}). The purpose of this work is to formaly show that despite the widespread use of the notion of criticality in non-convex optimization, the vast majority of cases encountered relate to the stronger notion of Fr\'echet stationarity, aligning formal guaranties with practical observations.

We first show that if $f$ and $C$ are assumed to be semi-algebraic, then considering the functions $\{f_v \colon x \mapsto f(x) + \left\langle v, x\right\rangle\}$, generically in $v$, all critical points of $f_v$ on $C$ are Fr\'echet stationary. This result illustrates the fact that the existence of critical points that are not Fr\'echet stationary is the consequence of a bad alignment of the objective function $f$ and the constraint set $C$, which is very unlikely under rigidity assumptions, modulo potential small perturbations of the loss function. This result is stated in the semi-algebraic setting, which encompasses many practical scenarios, including sparse vectors, bounded rank matrices. The same result holds for broader classes of functions and constraint sets, definable in o-minimal structures, but we do not expand on this and stick to the semi-algebraic setting for simplicity. This result is known, for example, it is a consequence of \cite[Corollary 4.3]{drusvyatskiy2016generic} and we provide a self contained argument for constrained optimization.

Second, we consider the well known projected gradient algorithm. A typical feature is that the resulting sequences tend to be attracted by critical points and we consider the question of Fr\'echet stationarity for these limit points. It turns out that the answer is positive, the projected gradient algorithm produces sequences that are attracted by Fr\'echet stationary points, with an explicit global quadratic estimate of the negative variation remainder. Although simple, this result provides a much stronger variational guaranty for the resulting limit points compared to mere criticality. This is an intuitive result, which is informaly known and has been investigated in more details recently \cite{olikier2024projected}. 

\subsection{Bibliographic review}

\textbf{Sparsity and rank constrained optimization:} The phenomenon described above does not affect convex constraint sets $C$ or more generally Clarke regular constraint sets $C$ see for example \cite[Definition 6.4]{rockafellar1998variational}. This includes constraint sets defined by smooth inequalities in nonlinear programming, under qualification conditions \cite[Theorem 6.14]{rockafellar1998variational}. The most well known applications where this property fails involve cardinality constraints, such as sparsity and rank. The sparse setting was largely studied in \cite{beck2013sparsity} with a carefull analysis of optimality conditions and algorithms, which were extended in \cite{beck2016minimization}. For low-rank matrices, it was remarked in \cite{luke2013prox} that rank deficiency is the source of an absence of (Clarke) regularity \cite[Definition 6.4]{rockafellar1998variational}, with a potential detrimental effect on the interpretation of optimality conditions \cite{hosseini2019tangent}.

The consequences of lack of regularity on optimality conditions, and optimality measures, was further studied in \cite{levin2023finding} for low-rank matrices under the name ``apocalypses'' with a very similar flavor as \cite{beck2013sparsity} for sparsity constraints. This constitutes further motivations to develop algorithmic schemes for low-rank matrix optimization that avoid this pathology and are attracted by stationary points in \cite{levin2023finding} followed by \cite{hou2021asymptotic,olikier2022apocalypse,olikier2023apocalypse,olikier2023first}. 
As mentioned in \cite{olikier2023gauss}, the absence of regularity only has rare consequences in practice and our main motivation is to provide formal guaranties for this observation in the form of genericity results on problem data and convergence guaranties for the projected gradient algorithm.

\textbf{Genericity in tame optimization:} Our first main result, \Cref{prop:genericFrechet}, relates to semi-algebraicity or tameness of the considered objective function $f$ and the constraint set $C$. Studying non-convex optimization and first order methods under such rigidity assumptions has a long history in optimization. Indeed, semi-algebraicity has numerous structural consequences on the optimization losses \cite{bolte2007loja,bolte2007clarke,bolte2009tame,ioffe2009invitation}. Furthermore, virtually all losses met in an optimization context are covered by tameness assumptions, see the numerous examples in \cite{attouch2010proximal,attouch2013convergence}, and the connection with deep learning in \cite{bolte2020mathematical,bolte2021conservative}. One can take advantage of these properties, for which semi-algebraicity is a mild sufficient condition, to analyse optimization algorithms and optimization landscapes. Examples include sequential convergence of deterministic optimization algorithms \cite{attouch2009convergence,attouch2010proximal,attouch2013convergence,bolte2014proximal,bolte2016majorization,pauwels2016value}, as well as the analysis of stochastic first order methods \cite{davis2020stochastic,bianchi2022convergence,bolte2021conservative,bolte2020mathematical,bolte2022subgradient}.

Genericity is a notion that is used to express the fact that a certain behavior is typical. It is most often expressed in measure theoretic terms (Lebesgue almost everywhere), or topological terms (residual sets are countable intersections of sets with dense interior). In general, the two notions do not coincide (see for example \cite[Theorem 1.6]{oxtoby1971measure}), but in the semi-algebraic setting they coincide and sometimes correspond to a stronger notion: being the complement of the union of finitely many lower dimensional embedded smooth manifolds. This is essentially due to the stratification property \cite[4.8]{van1996geometric}. Genericity results in an optimization context relate to the typical structure of the data of optimization problems \cite{daniilidis2009generic,bolte2011generic,daniilidis2011continuity,pham2016genericity,drusvyatskiy2016generic,bolte2017perturbed,lee2017generic} and generic desirable properties of optimization methods \cite{nie2014optimality,bianchi2023stochastic,davis2022proximal,davis2022nearly}. The first result of this known falls in this category, we show that for a generic semi-algebraic $f$ and a fixed semi-algebraic set $C$, there is no critical point that is not Fr\'echet stationary for the resulting constrained minimization problem. A consequence of this result is that for a generic smooth semi-algebraic function $f$ and a fixed closed set $C$, the ``apocalypses'' described in \cite{levin2023finding} do not exist. As mentioned earlier this result is a consequence of \cite[Corollary 4.3]{drusvyatskiy2016generic}, see also \cite[Theorem 9.60]{ioffe2017variational}. 

\textbf{Projected gradient algorithm:} Our second main result, \Cref{th:projGrad}, concerns the projected gradient algorithm proposed independently by Goldstein \cite{goldstein1964convex} and Levitin Polyak \cite{levitin1966constrained} for convex optimization with subsequent contributions in the convex setting \cite{bertsekas1976goldstein,calamai1987projected,dunn1981global,dunn1987convergence}. Our analysis is a consequence of a detailed analysis of the proximal gradient algorithm, the proximal mapping generalizing the projection. Convergence of the proximal point algorithms in a non-convex setting was considered in \cite{spingarn1982submonotone,kaplan1998proximal,attouch2009convergence}, convergence of the proximal gradient algorithms under semi-algebraic assumptions was given in \cite{attouch2013convergence}. The projected gradient algorithm was explicitely studied in \cite{beck2013sparsity} for sparsity constraints in relation to the notion of $L$-stationarity. Issues regarding convergence guaranties of the projected gradient algorithm is related to the concerns raised in \cite{levin2023finding} and as mentioned above, a similar question was studied in \cite{olikier2024projected}.

\subsection{Notations}
Throughout the paper, the ambient space is $\RR^p$.
We denote by $\left\langle \cdot,\cdot \right\rangle$ and $\|\cdot\|$, the Euclidean scalar product and Euclidean norm. We denote a set-valued map $F$, from $\RR^p$ to subsets of $\RR^p$ with the notation $F \colon \RR^p \rightrightarrows \RR^p$. For a subset $C \subset \RR^p$, we denote by $T_C, \hat{N}_C, N_C$ the tangent, regular normal and normal cones respectively, which are seen as set-valued maps $\RR^p \rightrightarrows \RR^p$ with empty values outside $C$. Relevant definitions are introduced along the paper.

\section{Main results}
We introduce the required elements of variational geometry in \Cref{sec:definitions} and state our two main results in \Cref{sec:generic} and \Cref{sec:projGrad}.

\subsection{Notions of stationarity}
\label{sec:definitions}
We use the notations and denominations of \cite{rockafellar1998variational}. First recall the definitions of the objects of interest.
\begin{definition}[Tangent and Normal Cones] 
	\label{def:tangentCone}
	For $x \in C$, $w \in \RR^p$ is an element of the tangent cone of $C$ at $x$, written $w \in T_C(x)$ if
	\begin{align*}
		\frac{x_k - x}{\tau_k} \to w
	\end{align*}
	for some sequence $(x_k)_{k \in \NN}$, in $C$ and $(\tau_k)_{k \in \NN}$ in $\RR_+$ decreasing to $0$.
	Furthermore, $v \in \RR^p$ is an element of the regular normal cone of $C$ at $x$, written $w \in \hat{N}_C(x)$ if
	\begin{align*}
		\left \langle v, y -x \right\rangle \leq o(\|y-x\|),\qquad y \in C,
	\end{align*}
	where the inequality is understood as $\lim\sup_{y \to x} \frac{\left \langle v, y -x \right\rangle}{\|y-x\|} \leq 0$. 
	Finally, $v \in \RR^p$ is an element of the normal cone of $C$ at $x$, written $w \in N_C(x)$ if 
	\begin{align*}
		\exists (x_k)_{k \in \NN},\, (v_k)_{k \in \NN},\, x_k \in C,\, v_k \in \hat{N}_C(x_k),\, k \in \NN,\, x_k \to x,\, v_k \to w,\, k \to \infty.
	\end{align*}
	$T_C, \hat{N}_C$ and $N_C$ can be seen as set-valued maps $\RR^p \rightrightarrows \RR^p$ by assigning empty values for $x \not \in C$.
\end{definition}

We gather known facts about these cones, the following is taken from Theorem 6.12 and 6.28 \cite{rockafellar1998variational}.
\begin{proposition}
	\label{prop:tangentNormal}
	Let $C \subset \RR^p$ be closed, then for all $x \in C$, $T_C(x)$ is a closed cone and $\hat{N}_C(x)$ is the polar of $T_C(x)$: $\hat{N}_C(x) = \{v \in \RR^p, \, \left\langle v,w \right\rangle \leq 0,\, \forall w \in T_C(x)\}$.

	Let $f \colon \RR^p \to \RR$ be $C^1$. Suppose that $x \in C$ is a local minimum of $f$ restricted to $C$, then the two equivalent conditions hold:
	\begin{align}
		\label{eq:frechetStationary}
		-\nabla f(x) &\in \hat{N}_C(x) \nonumber\\
		\mathrm{proj}_{T_C(x)} (-\nabla f(x)) &= 0.
	\end{align}
	A point $x \in C$ satisfying \eqref{eq:frechetStationary} is called Fr\'echet stationary for $f$ on $C$, which is equivalent to,
	\begin{align}
		\label{eq:illustrFrechet}
		f(y) - f(x) &\geq o(\|y-x\|),& y\in C.
	\end{align}
	This implies the stronger condition $-\nabla f(x) \in N_C(x)$, an $x \in C$ satisfying this condition is called critical for $f$ on $C$.
\end{proposition}

\Cref{prop:tangentNormal} suggests to use \eqref{eq:frechetStationary} as an optimality condition for constrained optimization, however the proposed quantity lacks basic continuity in general, which is troublesome for many applications. This motivates the introduction of the normal cone to $C$, $N_C$, which is the graph closure of the regular normal cone to $C$, recovering some form of continuity and the possibility to pass to limits. 
Typical optimization results fall in this scope and provide guaranties in terms of criticality, $-\nabla f(x) \in N_C(x)$, in the context of \eqref{eq:mainProblem}, which is necessary but not sufficient for \eqref{eq:frechetStationary}. While this constitutes a bona fide optimality condition, in the sense that if it is not satisfied, $x$ is not a local extremum, it may result in meaningless notion of criticality contrary to the interpretation of Fr\'echet stationarity in \eqref{eq:illustrFrechet}.

\begin{example}
	\label{ex:sparseVectors}
	Let $C \subset \RR^2$ be the set of $1$-sparse vectors, then $N_C(0,0) = T_C(0,0) = C$ while $\hat{N}_C(0,0) = \{0\}$. Set $f\colon (x,y) \mapsto (x-1)^2 + y^2$, then $(0,0)$ is critical for $f$ on $C$ but this does not have much variational meaning since $f$ has directional derivative $2$ in the $y$ direction, which is in $T_C(0,0)$ and is actually admissible with respect to the constraint induced by $C$.
\end{example}	
To ellaborate on this remark and illustrate the relevance of the normal cone in comparison to the regular normal cone, we quote Rockafellar and Wets \cite{rockafellar1998variational} regarding the phenomenon presented in \Cref{ex:sparseVectors}:
\begin{center}
	\textit{This possibility causes some linguistic discomfort over `normality', but the cone of such limiting normal vectors comes to dominate technically in formulas and proofs, $[\ldots]$ Many key results would fail if we tried to make do with regular normal vectors alone.}
\end{center} 
This absence of regularity is related to the main motivations in \cite{levin2023finding} to propose algorithms that do not suffer from it.

\subsection{Genericity of Fr\'echet stationarity}
\label{sec:generic}
We start by introducing the necessary tools from semi-algebraic geometry. An introduction to semi-algebraic and tame geometry is found in \cite{coste2000introduction,coste2000introductionSA} and a comprehensive overview is given in \cite{van1996geometric}, see also \cite{van1998tame}. We recall all the required concepts with necessary bibliographic pointers.
\subsubsection{Semi-algebraic geometry}
Let us first introduce the required definitions.
\begin{definition}
	Let $p,q \in \NN$ be arbitrary.

	A basic semi-algebraic set $S \subset \RR^p$ is the solution set to a polynomial system of inequalities.
	\begin{align}
		S = \{x \in \RR^p,\, P(x) = 0, Q_1(x)>0,\ldots Q_m(x) > 0\}
	\end{align}
	where $P,Q_1,\ldots,Q_m$ are polynomials of $p$ variables and $m \in \NN$.

	A semi-algebraic set is the finite union of basic semi-algebraic sets. 

	A function $f \colon \RR^p \to \RR^q$ is semi-algebraic if its graph $\{(x,z) \in \RR^{p + q},\, z = f(x)\}$ is semi-algebraic.

	A set-valued map $F \colon \RR^p \rightrightarrows \RR^q$ is semi-algebraic if its  graph $\{(x,z) \in \RR^{p + q},\, z \in F(x)\}$ is semi-algebraic.
	\label{def:semiAlgebraicSet}
\end{definition}

\begin{example}[semi-algebraic functions]
	Euclidean norm, square root, quotients, rational powers, matrix rank are semi-algebraic functions. Semi-algebraic functions are  closed under composition.
\end{example}

Semi-algebraic objects are closed under many relevant operations: intersection, unions, complementation and Cartesian product \cite{coste2000introduction}. The Tarski-Seidenberg principle \cite[Theorem 2.3]{coste2000introductionSA} allows to characterize semi-algebraic sets as the smallest o-minimal structure \cite[Exercise 1.17]{coste2000introduction}. Therefore the general tools of o-minimal geometry apply to the semi-algebraic setting. Furthermore, many results for semi-algebraic sets naturaly extend to the o-minimal setting. In this spirit, we gather below properties that will be useful in order to prove our genericity result.
\begin{proposition} Let $p,m \in \NN$ be arbitrary and $C \subset \RR^p$ and $F \colon \RR^p \rightrightarrows \RR^m$ be semi-algebraic.
	\begin{enumerate}
		\item The projection of $C$ onto a subspace is semi-algebraic. \label{prop:it1:TarskiSeidenberg}
		\item $T_C$, $\hat{N}_C$ and $N_C$ are semi-algebraic set-valued maps. \label{prop:it2:coneSA}
		\item The interior and closure of $C$ are semi-algebraic. \label{prop:it3:closure}
		\item If $F$ is a differentiable function (single valued), then its Jacobian is a semi-algebraic function. \label{prop:it4:derivative}
		\item The image of $C$ by $F$, $F(C) \subset \RR^m$ is semi-algebraic. \label{prop:it5:image}
		\item $C$ can be partitioned into a finite union of disjoint semi-algebraic smooth embedded submanifolds. \label{prop:it6:stratification}
	\end{enumerate}
	\label{prop:SA}
\end{proposition}

\begin{proof}[of \Cref{prop:SA}]
	These are well known and can be found in \cite{van1996geometric,coste2000introductionSA,coste2000introduction}, we provide proof arguments and detailed pointers for completeness.
	\begin{enumerate}
		\item This is Tarski-Seidenberg Theorem, up to a rotation, see \cite[Theorem 2.3]{coste2000introductionSA}. An equivalent formulation of this result is \cite[Theorem 2.6]{coste2000introductionSA} states that every first-order formula (quantification on variables), involving semi-algebraic sets or functions, polynomials, inequalities, equalities and the logical negation, conjunction and disjunction, describes a semi-algebraic object. In the following, we describe each set of interest with such a first-order formula, which implies that they are semi-algebraic (see \cite[Section 2.1.2]{coste2000introductionSA} and \cite[Theorem 1.13]{coste2000introduction}). 
		\item 
			\begin{align*}
				z \in T_C(x) \quad &\Leftrightarrow \quad \forall \epsilon > 0,\, \exists y \in C,\, y\neq x,\,  \|y-x\| \leq \epsilon,\, \left\vert \frac{y-x}{\|y-x\|} - \frac{z}{\|z\|} \right\vert \leq \epsilon\\
				z \in \hat{N}_C(x) \quad &\Leftrightarrow \quad \forall w \in T_C(x),\, \left\langle z,w\right\rangle \leq 0 \\
				z \in N_C(x) \quad &\Leftrightarrow \quad \forall \epsilon > 0, \exists y \in C,\, \exists v \in \hat{N}_C(y),\, \|x-y\|\leq \epsilon,\, \|v - z \| \leq \epsilon.
			\end{align*}
		\item 
			\begin{align*}
				x \in \mathrm{int}\ C\quad &\Leftrightarrow \quad \exists \epsilon > 0,\, \forall y \in \RR^p,\, \|x-y\|> \epsilon \text{ or } y \in C\\
				x \in \mathrm{cl}\ C\quad &\Leftrightarrow \quad \forall \epsilon > 0, \exists y \in C, \|y-x\| \leq \epsilon.
			\end{align*}
		\item 
			\begin{align*}
				M = J_F(x) \quad \Leftrightarrow \quad &M \in \RR^{m \times p},\, \forall \epsilon > 0,\, \exists \delta > 0,\, \forall y \neq x,\, \\
				&\|y-x\| > \delta \text{ or } \frac{\|F(y) - F(x) - M(y-x)\|}{\|y-x\|} \leq \epsilon
			\end{align*}
		\item 
			\begin{align*}
				z \in F(C) \quad \Leftrightarrow \quad \exists x \in C,\, z \in F(x).
			\end{align*}
		\item This is the geometric notion of stratification, see for example in  \cite[Claim 4.8]{van1996geometric}.
	\end{enumerate}
\end{proof}

Semi-algebraic sets come with a notion of integral dimension, denoted by $\mathrm{dim}\ C$ for a semi-algebraic set $C$, which agrees with the classical notion of dimension for affine sets or embedded manifolds.
The following facts can be found in \cite[Proposition 3.17, Theorem 3.22]{coste2000introduction} and will be useful to prove our genericity result.
\begin{proposition}
	Let $p,m \in \NN$ be arbitrary. 
	\begin{enumerate}
		\item If $ B\subset A \subset \RR^p$ are semi-algebraic, then $\mathrm{dim} B \leq \mathrm{dim} A$.
		\item If $A,B \subset \RR^p$ are semi-algebraic, then $\mathrm{dim}\ A \cup B = \max\{\mathrm{dim}\ A, \mathrm{dim}\ B\}$.
		\item For any $A \subset \RR^p$ semi-algebraic, $\mathrm{dim}\ \mathrm{cl}\ A = \mathrm{dim}\ A$, $\mathrm{dim}\ \mathrm{cl}\ A \setminus A < \mathrm{dim}\ A$. 
		\item For $f \colon \RR^p \to \RR^m$ and $A \subset \RR^p$, both semi-algebraic, $\mathrm{dim}\ f(A) \leq \mathrm{dim}\ A$. 
	\end{enumerate}
	\label{prop:dimension}
\end{proposition}

\subsubsection{Main result}

The following is our first main result. It is stated in the semi-algebraic setting, but it can be extended to functions and sets that are definable in the same o-minimal structure as the arguments rely on the elements described in \Cref{prop:SA} and \Cref{prop:dimension}, which hold for definable functions and sets \cite{coste2000introduction,van1996geometric}. 

\begin{theorem}
	\label{prop:genericFrechet}
	Let $f \colon \RR^p \to \RR$ be continuously differentiable and $C \subset \RR^p$ be closed, both semi-algebraic. Then there is $V \subset \RR^p$, a finite union of semi-algebraic embedded manifolds of dimension at most $p-1$, such that for all $v \not\in V$, all critical points of $f_v \colon x \mapsto f(x) + \left\langle v, x \right\rangle$ on $C$ are Fr\'echet stationary.
\end{theorem}
\begin{proof}[of \Cref{prop:genericFrechet}]
	By \Cref{prop:SA} \cref{prop:it2:coneSA}, $\hat{N}_C$ and $N_C$ are semi-algebraic.

	We work with $\hat{N}_C \colon C \rightrightarrows \RR^p$. It follows from \Cref{def:tangentCone} that for any $S \subset C$, we have for all $x \in S$, $\hat{N}_C(x) \subset \hat{N}_S(x)$. We may consider a partition of $C$ into $M_1,\ldots, M_m$ embedded smooth manifolds by \Cref{prop:SA} \cref{prop:it6:stratification}. For each $i=1,\ldots,m$, we have $M_i \subset C$ and therefore $\hat{N}_C(x) \subset \hat{N}_{M_i}(x)$. But $\hat{N}_{M_i}(x)$ is simply the normal space of $M_i$ at $x$ as described by differential geometry. Therefore the graph of $\hat{N}_C$ restricted to $M_i$ is contained in the normal bundle of $M_i$, which can be seen as an embedded submanifold of $\RR^{2p}$ of dimension $p$, see for example \cite[Theorem 6.23]{lee2012smooth}. Therefore, the graph of the restriction of $\hat{N}_C$ to $M_i$ is of dimension at most $p$ by \Cref{prop:dimension} item 1. The graph of $\hat{N}_C$ is the union of its restriction to $M_i$, $i=1,\ldots,m$ and it is therefore of dimension at most $p$ by \Cref{prop:dimension} item 2.
	
	Set $G$ the closure of $\mathrm{graph}\ \hat{N}_C$ in $\RR^{2p}$, it is semi-algebraic by \Cref{prop:SA} \cref{prop:it3:closure}. We have that $(x,z) \in G$ if and only if there is a sequence $x_k \to x$ and $z_k \to z$ such that $z_k \in \hat{N}_C(x_k)$ for all $k \in \NN$. In other words, we have $G = \mathrm{graph}\ N_C$. By \Cref{prop:dimension} item 3, the semi-algebraic set $H = \mathrm{cl}(\mathrm{graph}\ \hat{N}_C) \setminus \mathrm{graph}\ \hat{N}_C \subset \RR^{2p}$ has dimension at most $p-1$. The set $H$ can be understood as the graph of the possibly empty-valued map $S \colon x \rightrightarrows N_C(x) \setminus \hat{N}_C(x)$.
	
	Now consider the set-valued map $R \colon x \rightrightarrows S(x) + \nabla f(x)$. Using \Cref{prop:dimension} item 4 the dimension of $\mathrm{graph}\ R$ is at most $p-1$ because it is the image of $H$, by the map $(x,z) \mapsto (x,z+\nabla f(x))$, which is semi-algebraic by \Cref{prop:SA} \cref{prop:it4:derivative}. Now we have the following equivalence, for any $v \in \RR^p$
	\begin{align*}
		\exists x \in C,\, - \nabla f(x) - v \in N_C(x) \setminus \hat{N}_C(x) \qquad \Leftrightarrow \qquad \exists x \in C, v \in R(x)
	\end{align*}
	so that
	\begin{align*}
		\left\{ v \in \RR^p, \exists x \in C,\, -\nabla f(x) - v \in N_C(x) \setminus \hat{N}_C(x)\right\} \qquad = \qquad \mathrm{proj}_v \,\, \mathrm{graph}\ R.
	\end{align*}
	where $\mathrm{proj}_v (x,z) = z$ for any $x,z \in \RR^p$.  Setting $V =  \mathrm{proj}_v \,\, \mathrm{graph}\ R$, we have that $\mathrm{dim} V \leq p-1$  by \Cref{prop:dimension} item 4. By \Cref{prop:SA} \cref{prop:it6:stratification}, $V$ is a finite union of semi-algebraic embedded submanifolds of dimension $p-1$ at most by \Cref{prop:dimension} item 2.
\end{proof}

\begin{remark}
	The result of \Cref{prop:genericFrechet} holds generically in $v$, as understood in both measure theoretic terms (almost everywhere), or topological terms (residual). We perturb $f$ using a linear form, but one could consider a peturbation of the form $x \mapsto \epsilon \|x - c\|^2$ for small $\epsilon>0$, and the same result would hold generically in $c \in \RR^p$. It is easy to see that the critical point example in \Cref{ex:sparseVectors} would not persist under generic perturbation and \Cref{prop:genericFrechet} provides a general ground for this observation.
	\label{rem:perturbSquare}
\end{remark}

\subsection{Projected gradient is attracted by Fr\'echet stationary points}
\label{sec:projGrad}
Given a non-empty closed set $C \subset \RR^p$, the projection of $x \in \RR^p$ on $C$, denoted by $\mathrm{proj}_C(x)$ is given by the non-empty set
\begin{align*}
	\mathrm{proj}_C(x) = \arg\min_{z \in C} \|x - z\|.
\end{align*}
Given an initial point $x_0 \in C$ and a step-size parameter, $\gamma > 0$, the projected gradient algorithm iterates
\begin{align}
	\label{eq:projectedGradient}
	x_{k+1} \in \mathrm{proj}_C(x_k - \gamma \nabla f(x_k)).
\end{align}
The following is our second main result. It is a consequence of the analysis of the proximal gradient algorithm proposed in \Cref{sec:proxGrad}
\begin{theorem}
	\label{th:projGrad}
	Let $f \colon \RR^p \to \RR$ be $C^1$ with $L$-Lipschitz gradient and $C\subset \RR^p$ be non-empty and closed. Then for any step size $\gamma< 1/L$, any accumulation point of the projected gradient algorithm, $\bar{x}$, is Fr\'echet stationary such that 
	\begin{align}
		f(y) &\geq f(\bar{x}) - \frac{1}{\gamma}\|y-\bar{x}\|^2, & \forall y \in C, \label{eq:quantitativeFrechetProj}\\
		\mathrm{proj}_C(\bar{x} - s \nabla f(\bar{x})) &= \{\bar{x}\},& \forall 0<s<\gamma.\nonumber 
	\end{align}
	Furthermore, if such an accumulation point exists, $\mathrm{proj}_{T_C(x_k)}(- \nabla f(x_k)) \to 0$ as $k \to \infty$.
\end{theorem}	
\begin{proof}
	From \cite[Exercise 8.14]{rockafellar1998variational}: $\hat{\partial} \delta_C(\bar{x}) = \hat{N}_C(x)$, where $\delta_C$ is the indicator function of $C$ with value $0$ on $C$ an $+ \infty$ outside. Note that $\delta_C$ satisfies the hypotheses of \Cref{lem:proxUSC} for any $\delta > 0$.

	We have for any $\gamma < 1/L$ that $\tilde{f} = \gamma  f$ has $L\gamma < 1$ Lipschitz gradient and the proximal gradient algorithm with unit step on $\tilde{f}$ and $g = \delta_C$ in \eqref{eq:proxGrad} is equivalent to the projected gradient algorithm on $f$ with step size $\gamma$ so that \Cref{prop:proxGrad} applies. We obtain that all accumulation points $\bar{x}$ are Fr\'echet stationary such that $\bar{x} \in \mathrm{proj}_C( \bar{x} - \gamma \nabla f(\bar{x}))$ from \Cref{prop:proxGrad}, which means that $\mathrm{dist}(\bar{x} - \gamma \nabla f(\bar{x}),C) =  \gamma \| \nabla f(\bar{x})\|$. The quantitative statement on Fr\'echet stationarity follows from \Cref{prop:proxGrad} applied to $\tilde{f} = \gamma f$. 

	Let us prove unicity of the projection \cite[Example 6.16]{rockafellar1998variational}, fix $0<s<\gamma$. The case $\nabla f(\bar{x}) = 0$ is obvious so let us eliminate it. Denote by $B_1$ the ball of center $\bar{x} - \gamma\nabla f(\bar{x})$ and radius $\gamma\|\nabla f(\bar{x})\|$ and $B_2$ the ball of center $\bar{x} - s\nabla f(\bar{x})$ and radius $s\|\nabla f(\bar{x})\|$. Let us show that for any $x \in B_2$, $x \neq \bar{x}$, we have 
	\begin{align*}
		\|x - \bar{x} + \gamma\nabla f(\bar{x})\| &= \|x - \bar{x} + s \nabla f(\bar{x}) + (\gamma - s) \nabla f(\bar{x})\| \\
		&< s\|\nabla f(\bar{x})\| + (\gamma - s) \|\nabla f(\bar{x})\| = \gamma \|\nabla f(\bar{x})\| 
	\end{align*}
	where the strict inequality is from the triangle inequality. Indeed, either the triangle inequality is strict, or $x - \bar{x} + s \nabla f(\bar{x}) = \alpha (\gamma - s) \nabla f(\bar{x})$ for some $\alpha \geq 0$. In this second case, since $x \in B_2$, by taking the norm, we obtain $s \geq \alpha(\gamma - s)$, so that $x = \bar{x} - t \nabla f(\bar{x})$ where $t = s - \alpha (\gamma - s)\geq 0$ and $t \leq s$. The case $t = 0$ is excluded because we assumed that $x \neq \bar{x}$ and we have $0<t\leq s < \gamma$, so that
	\begin{align*}
		\|x - \bar{x} + \gamma\nabla f(\bar{x})\| = \|(\gamma - t) \nabla f(\bar{x})\| = |\gamma - t|\ \|\nabla f(\bar{x})\| < \gamma \|\nabla f(\bar{x})\| 
	\end{align*}
	We have shown that any $x \in B_2$ different from $\bar{x}$ is actually in $\mathrm{int} B_1$ and therefore at positive distance from $C$, otherwise this would contradict $\mathrm{dist}(\bar{x} - \gamma \nabla f(\bar{x}), C) =  \| \nabla f(\bar{x})\|$. Since $\bar{x} \in C$, it is the unique element in $B_2 \cap C$, which proves unicity of the projection.
	
	We conclude regarding the last statement $$\mathrm{proj}_{T_C(x_k)}(- \nabla f(x_k)) \to 0$$ by combining the fact that 
	$\mathrm{dist}(- \nabla f(x_k), \hat{N}_C(x_k)) \to 0$ from \Cref{prop:proxGrad} and \Cref{lem:polarCones}. 
\end{proof}

\begin{lemma}
	\label{lem:polarCones}
	Let $T\subset \RR^p$ be a closed cone, not necessarily convex and $N$ be its polar, $N = \{v \in \RR^p,\, \left\langle w,v\right\rangle \leq 0,\, \forall w \in T\}$, then for any $x \in \RR^p$, $\|\mathrm{proj}_T(x)\| \leq \mathrm{dist}(x,N)$.
	\label{lem:projectionCone}
\end{lemma}
\begin{proof}
	Set $z = \mathrm{proj}_T(x)$, if $\|z\| = 0$, then there is nothing to prove. Assume that $\|z\| > 0$. Set $\tilde{T} = \{\lambda z,\, \lambda \geq 0\}$ and $\tilde{N}$ its polar, we have
	\begin{align*}
		\tilde{T} &\subset T,\qquad\qquad
		\tilde{N} \supset N ,\qquad\qquad
		\mathrm{proj}_{\tilde{T}}(x) = z.
	\end{align*}
	Both $\tilde{N}$ and $\tilde{T}$ are convex cones and by Moreau's identity \cite[Section 4.b]{moreau1965proximite}, we have $x = \mathrm{proj}_{\tilde{T}}(x) + \mathrm{proj}_{\tilde{N}}(x)$ so that
	\begin{align*}
		\|\mathrm{proj}_{\tilde{T}}(x)\| = \| x - \mathrm{proj}_{\tilde{N}}(x)\| = \mathrm{dist}(x, \tilde{N}) \leq \mathrm{dist}(x, N).
	\end{align*}
\end{proof}

\begin{remark}[Comments on \Cref{th:projGrad}]
	It was identified in \cite{beck2013sparsity} that for sparsity constraints, local minimizers need to be fixed point of the projected gradient algorithm (a condition termed $L$ stationarity) and the projection has to be univalued. \Cref{th:projGrad} shows that for general sets, the projected gradient algorithm will be attracted by such points, generalizing the result of \cite{beck2013sparsity} for the Iterative Hard Thresholding algorithm, as illustrated in \Cref{fig:numericalIllustr}. This result is related to the notion of proximal normals \cite[Example 6.16]{rockafellar1998variational}, the sequences are actually attracted by the set of points $\bar{x} \in C$ such that $-\nabla f(\bar{x})$ is a proximal normal of $C$ at $\bar{x}$. The last assertion in \Cref{th:projGrad} ensures that the so called ``serendipity'' phenomenon described in \cite[Definition 2.8]{levin2023finding} does not affect the projected gradient algorithm. If we assume in addition that $f$ and $C$ are semi-algebraic, then the sequence actually converges, as shown in \cite{attouch2013convergence}. Finally if $f$ is convex one can add a factor $\frac{1}{2}$ in front of the quadratic term in \eqref{eq:quantitativeFrechetProj}.
	\label{rem:PGD}
\end{remark}

\subsection{Numerical illustration}
We illustrate the relevance of the result of \Cref{th:projGrad}, first with the avoidance of a critical point that is not Fr\'echet stationary as in \Cref{ex:sparseVectors}, and second with the avoidance of bad local minima on a grid. These are illustrative toy examples, and in both cases the observed behavior could be justified with elementary dedicated arguments. Exploring consequences of \Cref{th:projGrad} in practical application will be a matter of future research.

\paragraph{Sparsity constraints}
We consider as in \Cref{ex:sparseVectors} the set $C$ of $1$-sparse vectors in $\RR^2$ and a loss function is $f \colon (x,y) \mapsto (x-1)^2+ y^2$ whose global minimum on the constraint set is $x=1, y=0$. We depict in \Cref{fig:numericalIllustr} the sequence generated by the projected gradient algorithm in \eqref{eq:projectedGradient} for various step sizes and initializations, representing both the gradient and the projection steps explicitly. The point $(0,0)$ is critical but not Fr\'echet stationary, none of the three sequences converges to this point. Instead, they all converge to the global minimum, illustrating the result of \Cref{th:projGrad}.

\paragraph{Nonlinear optimization on a grid}
We consider the problem of minimizing a convex quadratic function, where the constraint set is a regular grid in $\RR^2$. In this setting, all feasible points are local minimizers, hence Fr\'echet critical. Yet \Cref{th:projGrad} predicts that not all of them are attractors of the projected gradient algorithm. We illustrate this with several projected gradient sequences in \Cref{fig:numericalIllustr} displaying explicitely the points which do not satisfy the quantitative estimate \eqref{eq:projectedGradient} (with factor $\frac{1}{2}$ for convex functions, see \Cref{rem:PGD}). The sequences stop when they reach these stationary points as predicted by \Cref{th:projGrad}.

\begin{figure}[ht]
	\centering
	\includegraphics[width=.38\textwidth]{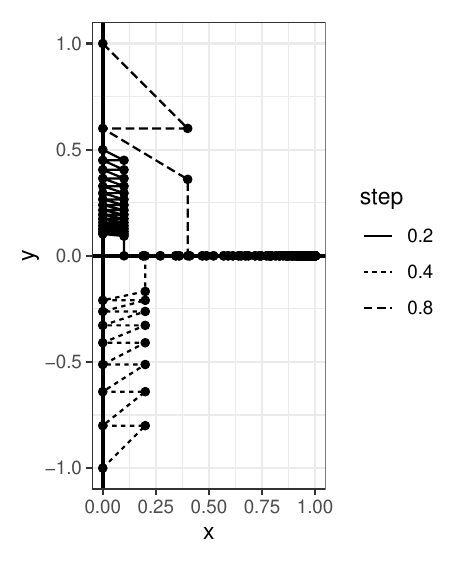}
	\includegraphics[width=.6\textwidth]{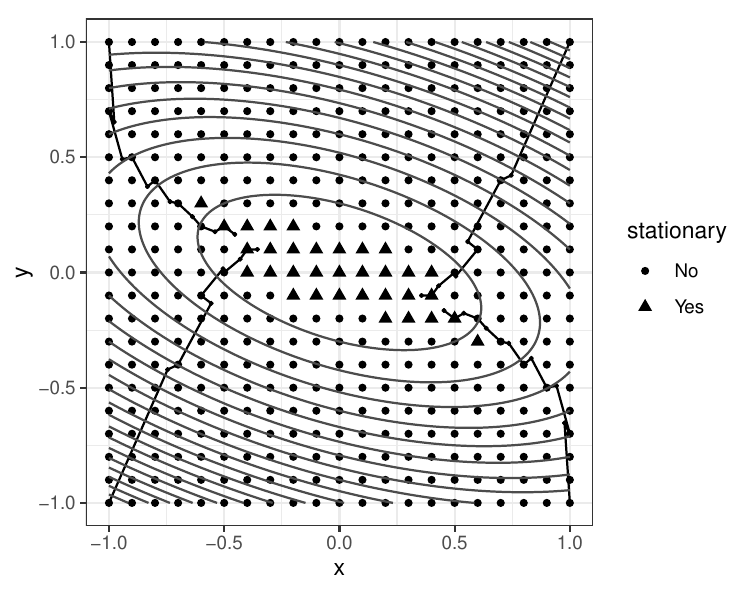}
	\caption{Left: Illustration of the avoidance of the non Clarke regular point of \Cref{ex:sparseVectors}, the constraint set is depicted by the thick black lines and the thiner lines display several projected gradient sequences with different step-sizes. Right: Avoidance of bad local minima. The feasible set is a grid and the contour of the convex quadratic objective is displayed. Every feasible point is a local minimum and the shape of the point indicate those that satisfy the quantitative estimate \eqref{eq:quantitativeFrechetProj} (with an additional factor $\frac{1}{2}$ from \Cref{rem:PGD}). We display several projected gradient sequences which all stop at the points satisfying these estimates. } 
	\label{fig:numericalIllustr}
\end{figure}

\section{The proximal gradient algorithm}
\label{sec:proxGrad}
In this section we provide a general result for the proximal gradient algorithm, from which \Cref{th:projGrad} follows. We first recall the necessary notations and concepts, they can be found in \cite{rockafellar1998variational}. This section can be seen of independent interest.
\subsection{Technical results from non-smooth analysis}
The following extends the notion of gradient in a natural way.
\begin{definition}[Regular subdifferential] 
	\label{def:regularSubdiff}
	Let $f \colon \RR^p \to \RR \cup \{ + \infty\}$ and consider $x \in \RR^p$ such that $f(x) < + \infty$. Then $v \in \hat\partial f(x)$ if 
	\begin{align*}
		f(y) \geq f(x) + \left\langle v, y - x \right\rangle + o(\|y-x\|).
	\end{align*}
	This notation means that $\lim\inf_{y \to x} \frac{f(y) - f(x) - \left\langle v, y - x \right\rangle}{\|y-x\|} \geq 0$.
\end{definition}
We obtain an optimality condition as a consequence of the definition in \cite[Theorem 10.1]{rockafellar1998variational}.
\begin{theorem}[Fermat Rule]
	If $x \in \RR^p$ is a local minimizer of a lower semicontinuous function $f \colon \RR^p \to \RR \cup \{ + \infty\}$, then $0 \in \hat \partial f(x)$.
	\label{th:fermatRule}
\end{theorem}
Conversely, a point $x \in \RR^p$ with $f(x)$ finite satisfying $0 \in \hat \partial f(x)$ has non-negative first-order variations around $x$, in the sense that $f(y) - f(x) \geq o(\|y-x\|)$. Such a point is called Fr\'echet critical. While calculus is in general out of scope for this type of object, it is possible to obtain sum rules when combined with a $C^1$ function \cite[Exercise 8.8 (c)]{rockafellar1998variational}.

\begin{lemma}[Smooth sum rule] 
	\label{lem:sumRule}
	Let $g \colon \RR^p \to \RR \cup \{ + \infty\}$ be lower semicontinuous and consider $x \in \RR^p$ such that $g(x) < + \infty$. Let $f \colon \RR^p \to \RR$ be $C^1$, then $\hat{\partial} (f + g) (x) = \hat{\partial} g(x) + \nabla f(x)$.
\end{lemma}
\begin{proof}
	From \cite[Corollary 10.9]{rockafellar1998variational} we have $\hat{\partial} (f + g) (x) \supset \hat{\partial} g(x) + \nabla f(x)$. Let us prove the reverse inclusion. Choose $v \in \hat{\partial} (f+g)(x)$, we have  by \Cref{def:regularSubdiff} and continuous differentiability.
	\begin{align*}
		\underset{y \to x}{\mathrm{liminf}} \quad\frac{f(y) + g(y) - f(x) - g(x) - \left\langle v , y - x \right\rangle}{\|y-x\|} &\geq 0 \\
		\lim_{y \to x} \quad \frac{f(y) - f(x) - \left\langle \nabla f(x), y - x \right\rangle}{\|y-x\|} &= 0.
	\end{align*}
	We deduce by a substraction that
	\begin{align*}
		\underset{y \to x}{\mathrm{liminf}}\quad \frac{g(y) - g(x) - \left\langle v - \nabla f(x), y - x \right\rangle}{\|y-x\|} &\geq 0,
	\end{align*}
	which shows that $v - \nabla f(x)\in \hat{\partial} g(x)$, which is the desired result.
\end{proof}

\subsection{The proximal gradient algorithm and Fr\'echet stationarity}
\label{sec:proxGrad}
Given a lower semicontinuous function $g \colon \RR^p \to \RR \cup \{+\infty\}$, the proximity operator of $g$ is defined as the possibly empty valued mapping
\begin{align*}
	\mathrm{prox}_g(x) = \arg\min_{y \in \RR^p} g(y) + \frac{1}{2} \|y-x\|^2.
\end{align*}
The following Lemma provides a sufficient condition for $\mathrm{prox}_g$ to be well behaved.
This is \cite[Theorem 1.25]{rockafellar1998variational}, we provide a short proof for completeness.
\begin{lemma}
	\label{lem:proxUSC}
	Let $g \colon \RR^p \to \RR \cup \{+\infty\}$ be lower semicontinuous, finite at least at one point, such that $g + \frac{1-\delta}{2} \|\cdot\|^2$ is bounded below for some $\delta > 0$. Then $\mathrm{prox}_g \colon \RR^p \rightrightarrows \RR^p$ has non-empty values, is locally bounded and upper semicontinuous, in the sense that for any converging sequences $y_k \in \mathrm{prox}_g(x_k)$, $k \in \NN$, $x_k \to x$, $y_k \to y$, we have $y \in \mathrm{prox}_g(x)$.
\end{lemma}
\begin{proof}
	By assumption $g + \frac{1}{2} \|\cdot\|^2$ is coercive so that the prox operator is compact valued and locally bounded. Let $(x_k)_{k \in \NN}$ and $(y_k)_{k \in \NN}$ be sequences such that $y_k \in \mathrm{prox}_g(x_k)$ for all $k \in \NN$, and $x_k \to x$, $y_k \to y$ as $k \to \infty$. For any $z \in \RR^p$, and any $k \in \NN$, we have
	\begin{align*}
		g(z) + \frac{1}{2} \|z- x_k\|^2 \geq g(y_k) + \frac{1}{2} \|y_k- x_k\|^2.
	\end{align*}	
	Hence for any $z \in \RR^p$, we have by lower semicontinuity
	\begin{align*}
		g(z) + \frac{1}{2} \|z- x_k\|^2 \geq {\mathrm{liminf}}_{k \to \infty} \ g(y_k) + \frac{1}{2} \|y_k- x_k\|^2 \geq g(y) + \frac{1}{2} \|y- x\|^2,
	\end{align*}	
	which is what we wanted to prove.
\end{proof}

We now state the main result of this section.
\begin{theorem}
	\label{prop:proxGrad}
	Let $f \colon \RR^p \to \RR$ be $C^1$ with $1-\delta$ Lipschitz gradient for some $\delta \in (0,1)$ and $g$ be as in \Cref{lem:proxUSC}. Fix $x_0 \in \RR^p$, $g(x_0)< +\infty$, and consider the recursion
	\begin{align}
		\label{eq:proxGrad}
		x_{k+1} \in \mathrm{prox}_g(x_k - \nabla f(x_k)).
	\end{align}
	Any accumulation point $\bar{x}$ of $(x_k)_{k \in \NN}$ is Fr\'echet stationary for $f+g$ such that 
	\begin{align}
		\bar{x} &\in \mathrm{prox}_g( \bar{x} - \nabla f(\bar{x})) \label{eq:quantitativeFrechetProx}\\
		f(y)+g(y) &\geq f(\bar{x}) + g(\bar{x}) - \|y - \bar{x}\|^2,& \forall y \in \RR^p. \nonumber 
	\end{align}
	Furthermore, $\mathrm{dist}(-  \nabla f(x_{k+1}),\hat{\partial} g (x_{k+1})) \to 0$ as $k \to \infty$.
\end{theorem}
\begin{proof}
	One can check that $x_{k+1} \in \arg\min_y f(x_k) + \left\langle \nabla f(x_k), y -x_k \right\rangle + \frac{1}{2} \|y - x_k\|^2 + g(y)$ by completing the square.
	
	Combining with the descent lemma for Lipschitz gradient functions \cite[Lemma 1.2.3]{nesterov2003introductory}, we have 
	\begin{align*}
		f(x_k) + g(x_k) &\geq f(x_k) + \left\langle \nabla f(x_k), x_{k+1} -x_k \right\rangle + \frac{1}{2} \|x_{k+1} - x_k\|^2 + g(x_{k+1}) \\
		& =f(x_k) + \left\langle \nabla f(x_k), x_{k+1} -x_k \right\rangle + \frac{1- \delta}{2} \|x_{k+1} - x_k\|^2 + g(x_{k+1}) + \frac{\delta}{2} \|x_{k+1} - x_k\|^2 \\
		&\geq f(x_{k+1}) + g(x_{k+1}) + \frac{\delta}{2} \|x_{k+1} - x_k\|^2.
	\end{align*}
	Now suppose that the sequence $(x_k)_{k \in \NN}$ has an accumulation point $\bar{x}$. In this case $f(x_k) + g(x_k)$ is decreasing, and it converges to a finite value as $\lim\inf_{k \to \infty} f(x_k) + g(x_k) \geq f(\bar{x}) + g(\bar{x})$ by lower semicontinuity. Therefore the increments $x_{k+1} - x_k$ tend to $0$ and, up to the same subsequence, both $x_k$ and $x_{k+1} \in \mathrm{prox}_g(x_k - \nabla f(x_k))$ tend to $\bar{x}$. Using \Cref{lem:proxUSC}, we have that $\bar{x} \in \mathrm{prox}_g( \bar{x} - \nabla f(\bar{x}))$ so that, using Fermat rule in \Cref{th:fermatRule} and \Cref{lem:sumRule}
	\begin{align}
		\bar{x} &\in \arg\min_y g(y) + \frac{1}{2} \|y - \bar x + \nabla f( \bar x)\|^2 \label{eq:fixedPointProximalGradient}\\
		0 &\in \hat \partial \left( g(y) + \frac{1}{2} \|y - \bar x + \nabla f( \bar x)\|^2 \right)_{y= \bar x} = \hat\partial g(\bar{x}) + \nabla f(\bar{x}),\nonumber 
	\end{align}
	which is Fr\'echet stationarity. This actually ensures that $-\nabla f(\bar{x})$ is a proximal subgradient of $g$, and using \cite[Proposition 8.46]{rockafellar1998variational} and the descent Lemma, for all $y \in \RR^p$,
	\begin{align}
		g(y) &\geq g(\bar{x}) + \left\langle - \nabla f(\bar{x}), y - \bar{x}\right\rangle - \frac{1}{2} \|y - \bar{x}\|^2 \nonumber\\
		f(y) &\geq f(\bar{x}) + \left\langle \nabla f(\bar{x}), y - \bar{x}\right\rangle - \frac{1}{2} \|y - \bar{x}\|^2 \label{eq:convex}\\
		f(y)+g(y) &\geq f(\bar{x}) + g(\bar{x}) - \|y - \bar{x}\|^2.\nonumber
	\end{align}
	For the last point, using Fermat rule in \Cref{th:fermatRule} for the prox operator leads to 
	\begin{align*}
		 x_{k+1} - x_k + \nabla f(x_k) &=    x_{k+1} - x_k + \nabla f(x_k) - \nabla f(x_{k+1}) + \nabla f(x_{k+1})  \\ 
		&\in -\hat{\partial} g (x_{k+1})  
	\end{align*}
	so that 
	\begin{align*}
		\mathrm{dist}(-  \nabla f(x_{k+1}),\hat{\partial} g (x_{k+1})) \leq \|x_{k+1} - x_k\| + \|\nabla f(x_k) - \nabla f(x_{k+1})\| \underset{k \to \infty}{\to} 0,
	\end{align*}
	which is the second result.
\end{proof}

\begin{remark}[Comments on \Cref{prop:proxGrad}]
	If in addition, the function $f$ and the set $C$ are assumed to be semi-algebraic, then the sequence actually converges \cite{attouch2013convergence}. The quadratic lower bound provides a quantitative estimate of Fr\'echet stationarity. Furthermore, if $f$ is convex, then one can add a factor $\frac{1}{2}$ in front of the quadratic term in \eqref{eq:quantitativeFrechetProx}, since the inequality \eqref{eq:convex} can be tightened. 
	\label{rem:proxGrad}
\end{remark}



\bibliographystyle{acm}
\bibliography{refs}

\end{document}